\documentclass[10 pt,reqno]{amsart}
\usepackage{amsmath}
\usepackage{amsfonts}
\usepackage{amssymb}
\usepackage{amsthm}
\usepackage{amscd}
\usepackage{a4wide}
\usepackage{enumerate}
\usepackage{dsfont} % i added for blackboard bold 1's
\usepackage{xcolor}
\usepackage{graphics}
\usepackage{eucal}
\usepackage{mathrsfs}
\usepackage{bm}
\usepackage{mathtools}
\mathtoolsset{showonlyrefs}
\usepackage{color}
\usepackage[pdftex,bookmarks,colorlinks,breaklinks]{hyperref}  % PDF hyperlinks, with coloured links
\usepackage{enumitem}
\hypersetup{linkcolor=red,citecolor=blue,filecolor=dullmagenta,urlcolor=darkblue} % coloured links
\newtheorem {theorem}    {Theorem}[section]

\newtheorem {corollary}  [theorem]    {Corollary}
\newtheorem {proposition}[theorem]    {Proposition}

\theoremstyle{definition}
\newtheorem {definition} [theorem]    {Definition}

\newtheorem {remark}    [theorem]    {Remark}

\newcounter{AbcT}

\numberwithin{equation}{section}

\newcommand{\IGNORE}[1]{}

\newcommand{\defeq}{\overset{\text{\tiny def}}{=}}

\renewcommand{\limsup}{\varlimsup}

\newcommand\degree{\operatorname{deg}}

\newcommand{\diag}{\operatorname{diag}}

\begin{document}
\title{Diophantine transference principle over function fields}

\begin{abstract} We study the Diophantine transference principle over function fields. By adapting the approach of Beresnevich and Velani to the function field set-up, we extend many results from homogeneous Diophantine approximation to the realm of inhomogeneous Diophantine approximation over function fields. This also yields the inhomogeneous Baker-Sprind\v{z}uk conjecture over function fields as a consequence. Furthermore, we prove the upper bounds for the general non-extremal scenario.

\end{abstract}

\author{Sourav Das}

\author{Arijit Ganguly}
 
\address{Department of Mathematics and Statistics, Indian Institute of Technology Kanpur, Kanpur, 208016, India}
\email{iamsouravdas1@gmail.com, arijit.ganguly1@gmail.com}

\maketitle
\section{Introduction}\label{section:intro}
 \noindent The  theory of metric Diophantine approximation on manifolds begins with Mahler's conjecture that says almost every point on the Veronese curve $\mathcal V_n := \{(x,x^2,\dots,x^n):x \in \mathbb R\}$ is not very well approximable. After three decades V. Sprind\v{z}uk (\cite{Spr1}) proved this conjecture in $1960$, and while proving that Sprind\v{z}uk conjectured (\cite{Spr2}) that the conclusion of Mahler's conjecture is true for any nondegenerate analytic submanifold of $\mathbb R^n$. This has further been strengthened by Baker and commonly referred to as the `Baker-Sprind\v{z}uk' conjecture in literature. Despite some partial results, the conjecture in its full generality remained unsolved for quite some time and finally it has been resolved by G.A. Margulis and D.Y. Kleinbock in $1998$ in the landmark paper \cite{KM}. They used techniques from homogeneous dynamics to solve this conjecture. Sooner rather than later people became interested in the inhomogeneous Diophantine approximation (\cite{BBV,BDV,BGGV,BV}). Beresnevich and Velani  observed that sometimes the results of homogeneous Diophantine approximation can be transferred to the inhomogeneous context by using certain transference principle. In \cite{BV}, Beresnevich and Velani proved many results of inhomogeneous Diophantine approximation, including the inhomogeneous Baker-Sprind\v{z}uk conjecture, using the aforementioned transference principle. In this paper, our aim is to explore this theme in the context of function fields. We first recall the function field set-up.\\

Consider the function 
field $\mathbb{F}_{q}(X)$, where $q= p^d$, $p$ is a prime and $d\in \mathbb{N}$. We define  a  non-archimedean  absolute value  $|\cdot|$
on $\mathbb{F}_{q}(X)$ in the following manner : 
\[ |0| := 0\,\,  \text{ and} \,\, \left|\frac{f}{g}\right| =e^{\displaystyle \degree f- \degree g}
\text{ \,\,\,for all } f, g\in \mathbb{F}_{q}[X] \setminus \{0 \}\,.\] 
By $\mathbb{F}_{q}((X^{-1})),$ we denote the field of Laurent series in $X^{-1}$ over the finite field $\mathbb{F}_q$. Note that $\mathbb{F}_{q}((X^{-1}))$ is the completion of $\mathbb{F}_{q}(X)$ with respect to the absolute value $|\cdot|$ and the absolute value of $\mathbb{F}_{q}(X)$ extends to an absolute value of  $\mathbb{F}_{q}((X^{-1}))$ as follows. 
For $f \in \mathbb{F}_{q}((X^{-1})) \setminus \{0\}$, we can write 
$$f=\displaystyle \sum_{\ell\leq \ell_{0}} {\alpha}_{\ell} X^{\ell} \text{ where }\ell_0 \in \mathbb{Z},\,{\alpha}_{\ell}\in \mathbb{F}_{q}\,\,\mbox{and}\,\, {\alpha}_{\ell_0}
\neq 0\,. $$

\noindent We define   
$|f|:=e^{\ell_0}$. With respect to this absolute value,  $\mathbb{F}_{q}((X^{-1}))$ is an ultrametric, complete and separable metric space. Recall that every local field of positive characteristic is isomorphic to some $\mathbb{F}_{q}((X^{-1}))$. \\

\noindent From now on  $\mathbb{F}_{q}[X]$ and $\mathbb{F}_{q}((X^{-1}))$ will be denoted by $\Lambda$ and $F$ respectively. We equipped  $F^n$ with the supremum norm in the following manner.
\[||\mathbf{y}|| := \displaystyle \max_{1\leq i\leq n} |y_i|\text{ \,\,for all \,} \mathbf{y}=(y_1,y_2,...,y_n)\in F^{n}\,,\]

\noindent and with the topology induced by the sup norm. Note that $\Lambda^n$ is discrete in $F^n$ and the local compactness of $F$ implies that $F^n$ is locally compact. By $\lambda$ we denote the Haar measure on $F^n$ such that $\lambda (\{\mathbf{x} \in F^n : \|\mathbf{x}\| \leq 1 \})=1$. \\

Diophantine approximation over function fields has been studied extensively of late, especially since Mahler developed the geometry of numbers in this context (\cite{M}). We refer the reader to the survey \cite{L1} and to \cite{DG,KN,RHSW} for the most recent developments in this area. Let us first recall Dirichlet's theorem over function fields from \cite[Theorem 2.1]{GG}.
\begin{theorem}\label{thm:Dirich}
    Let $m,n \in \mathbb N,$ $\ell=m+n$ and
    \begin{equation*}
        \alpha^+ := \left\{ \mathbf{t}:=(t_1,\dots,t_l) \in \mathbb{Z}_{+}^{\ell}: \sum_{i=1}^{m}t_i=\sum_{i=1}^{n}t_{m+i}\right\},
    \end{equation*}
where $\mathbb Z_+$ denotes the set of nonnegative integers. Let $Z$ be an $m \times n $ matrix over $F.$ Then for any $\mathbf{t} \in \alpha^+$, there exists solutions $\mathbf{q}=(q_1,\dots,q_n) \in \Lambda^n \setminus \{0\}$ and $\mathbf{p}=(p_1,\dots,p_m) \in \Lambda^m$ of the following system of inequalities:
\begin{equation}\label{eqn:dirichlet}\left \{ \begin{array}{rcl} |Z_{i}\mathbf{q}+p_i|< e^{-t_{i}} &\mbox{for} &i=1,2,...,m \\ |q_j|< e^{t_{m+j}}&\mbox{for} &j=1,2,...,n\,, \end{array} \right. 
\end{equation}
where $Z_1, \dots,Z_m$ are the row vectors of $Z.$
\end{theorem}
In the set-up of this paper, we will not pursue this much generality, instead, we will only consider unweighted Diophantine approximation. We denote the set of all $m \times n$ matrices over $F$ by $F^{m\times n}.$ Let us first recall the definition of exponent from \cite{GG-Sprind}.

\begin{definition}\label{def:inhomo_expo}
    
 Given $Y \in F^{m \times n} $ and $\boldsymbol{\theta} \in F^m,$ we define the \emph{inhomogeneous exponent} $\omega(Y,\boldsymbol{\theta})$ of $Y,$ as the supremum of real numbers $\omega \geq 0$ such that there exists a solution $(\mathbf{p}, \mathbf{q})  \in \Lambda^m \times (\Lambda^n \setminus \{0\})$ to the following system of inequalities 
\begin{equation}\label{equ:inhom_expo}
    \|Y \mathbf{q} + \mathbf{p} + \boldsymbol{\theta}\|^m < e^{- \omega T} \quad \text{and } \quad \|\mathbf{q}\|^n < e^T
\end{equation}
 for arbitrarily large $T \geq 1. $
 \end{definition}
  
   \begin{definition}\label{def:inhomo_unif_expo}
    Given $Y \in F^{m \times n} $ and $\boldsymbol{\theta} \in F^m,$ we define the \emph{inhomogeneous uniform exponent} $\hat{\omega}(Y,\mathbf{\boldsymbol{\theta}})$ of $Y,$ as the supremum of real numbers $\hat{\omega} \geq 0$ such that there exists a solution $(\mathbf{p}, \mathbf{q})  \in \Lambda^m \times (\Lambda^n \setminus \{0\})$ to the following system of inequalities 
\begin{equation}\label{equ:inhomo_unif_expo}
    \|Y \mathbf{q} + \mathbf{p} + \mathbf{\boldsymbol{\theta}}\|^m < e^{-\hat{\omega}  T} \quad \text{and } \quad \|\mathbf{q}\|^n < e^T
\end{equation}
 for all sufficiently large $T  \geq 1. $
   \end{definition}

\begin{definition}\label{def:mult_inhomo_expo}
    Given $Y \in F^{m \times n} $ and $\mathbf{\boldsymbol{\theta}} \in F^m,$ we define the \emph{multiplicative inhomogeneous exponent} ${\omega}^{\times}(Y,\mathbf{\boldsymbol{\theta}})$ of $Y,$ as the supremum of real numbers ${\omega} \geq 0$ such that there exists a solution $(\mathbf{p}, \mathbf{q})  \in \Lambda^m \times (\Lambda^n \setminus \{0\})$ to the following system of inequalities 
\begin{equation}\label{equ:mult_inhomo_expo}
   \prod (Y \mathbf{q} + \mathbf{p} + \mathbf{\boldsymbol{\theta}}) < e^{- {\omega}  T} \quad \text{and } \quad {\prod}_+ (\mathbf{q}) < e^T
\end{equation}
 for  arbitrarily large $T \geq 1, $ where
 \begin{equation*}
     \prod (\mathbf{y}):= \prod_{j=1}^{m} |y_j| \quad \text{and} \quad {\prod}_+ (\mathbf{q}):= \prod_{i=1}^{n} \max \{1,|q_i|\}
 \end{equation*}
 for $\mathbf{y}=(y_1,\dots,y_m) \in F^m$ and $\mathbf q=(q_1,\dots,q_n) \in \Lambda^n.$
\end{definition}
In a similar manner given any $Y \in F^{m \times n}$ and $\mathbf{\boldsymbol{\theta}} \in F^m$, the \emph{multiplicative inhomogeneous uniform exponent} $\hat{{\omega}}^{\times}(Y,\mathbf{\boldsymbol{\theta}})$ of $Y$ is defined.
Note that the analogous homogeneous exponents are provided by the special case of $\mathbf{\boldsymbol{\theta}}=\mathbf{\boldsymbol{0}},$ where $\mathbf{\boldsymbol{0}}=(0,\dots,0) \in F^m.$ In other words, given $Y \in F^{m \times n}$ we define the analogous homogeneous exponents as follows:
\begin{equation*}
    \omega(Y):=\omega(Y,\mathbf{\boldsymbol{0}}) \quad \text{and} \quad \hat{\omega}(Y):=\hat{\omega}(Y,\mathbf{\boldsymbol{0}}),
\end{equation*}
\begin{equation*}
    \omega^{\times}(Y):=\omega^{\times}(Y,\mathbf{\boldsymbol{0}}) \quad \text{and} \quad {\hat{\omega}}^{\times}(Y):={\hat{\omega}}^{\times}(Y,\mathbf{\boldsymbol{0}}).
\end{equation*}
Note that 
\begin{equation}\label{equ:mr1}
\omega^{\times}(Y,\boldsymbol{\theta}) \geq \omega(Y,\boldsymbol{\theta}) \text{ for all } Y \in F^{m \times n}, \boldsymbol{\theta} \in F^m.
\end{equation}
It is worth mentioning that by Dirichlet's theorem, we immediately get that $\omega(Y) \geq 1$ for all $Y \in F^{m \times n}.$ Hence $\omega^{\times}(Y) \geq 1.$ In fact an easy application of Borel-Cantelli lemma shows that $\omega (Y)=1$ and  $\omega^{\times}(Y)=1$ for $\lambda$ almost every $Y \in F^{m \times n}.$ We say that $Y \in F^{m \times n}$ is \emph{very well approximable} (VWA) if $\omega (Y) >1$ and $Y \in F^{m \times n}$ is \emph{very well multiplicatively approximable} (VWMA) if $\omega^{\times} (Y) >1.$

The function field analogue of Baker-Sprind\v{z}uk conjecture says that almost every point on an analytic nondegenerate submanifold $\mathcal M$ of $F^n,$ identified with either columns $F^{n \times 1}$ (in simultaneous Diophantine approximation) or rows $F^{1 \times n}$ (in dual Diophantine approximation), are not VWMA with respect to the natural measure on $\mathcal M.$ In the language of exponents it says the following. Let  $\mathcal M$  be any analytic nondegenerate submanifold of $F^n.$ Then
\begin{equation}\label{equ:sext}
    \omega^{\times}(Y)=1 \text{ for almost every } Y \in \mathcal M,
\end{equation}
and this clearly implies that
\begin{equation}\label{equ:ext}
    \omega(Y)=1 \text{ for almost every } Y \in \mathcal M.
\end{equation}
This was settled by A. Ghosh in \cite{G-pos}. The manifolds of $F^n$ which satisfy \eqref{equ:sext} and \eqref{equ:ext} are referred to as \emph{strongly extremal} and \emph{extremal} respectively. In view of the above discussion we have the following theorem:
\begin{theorem}[Theorem $3.7,$ \cite{G-pos}]\label{thm:G}\label{thm:G_pos}
    Let $\mathcal{M}$ be an analytic non-degenerate submanifold of $F^n.$ Then $\mathcal M$ is strongly extremal.
\end{theorem}
Kleinbock, Lindenstrauss and Weiss  revolutionised  the whole area by introducing the idea of measures being extremal rather than sets. Following their terminology, we say that a measure $\mu$ supported on a subset of $F^{m \times n}$ is \emph{extremal} (\emph{strongly extremal}) if $\omega(Y)=1$ ($\omega^{\times}(Y)=1$) for $\mu$-almost every point $Y \in F^{m \times n}$.

\noindent Due to Khintchine's transference principle, we say that a measure $\mu$ on $F^n$ is (strongly) extremal if it is (strongly) extremal as a measure on $F^{1 \times n}$ or $F^{n \times 1}.$ 
\begin{theorem}\label{thm:KLW_fun}
    Let $\mu$ be a friendly measure on $F^n$. Then $\mu $ is strongly extremal.
\end{theorem}
\begin{proof}
    Proof of the above theorem is analogous to that of Theorem 11.1 and Corollary 11.2 of \cite{KT}.
\end{proof}
We will define the notion of \emph{friendly measure} in Section \ref{section:cont_meas}. Friendly measures provide us with a large class of measures on $F^n$ that includes the natural measure on a nondegenerate analytic manifold. Hence we have
$$\text{Theorem \ref{thm:KLW_fun}}\Longrightarrow \text{Theorem \ref{thm:G_pos}}.$$

\begin{definition}\label{def:inhomo_ext}
    A measure $\mu$ supported on a subset of $F^{m \times n} $ is said to be \emph{inhomogeneously extremal} if for all $\boldsymbol{\theta} \in F^m$
    \begin{equation}
        \omega(Y,\boldsymbol{\theta})=1 \quad \text{for $\mu$-almost every }\,\, Y \in F^{m \times n}. 
    \end{equation}
We call $\mu$ \emph{inhomogeneously strongly extremal} if for all $\boldsymbol{\theta} \in F^m$
\begin{equation}
        \omega^{\times}(Y,\boldsymbol{\theta})=1 \quad \text{for $\mu$-almost every }\,\, Y \in F^{m \times n}.
    \end{equation}
\end{definition}
It is easy to see that this naturally generalised notion of extremality in the homogeneous case ($\boldsymbol{\theta}=\boldsymbol{0}$). In the homogeneous case, strong extremality implies extremality but it is not at all clear why an inhomogeneously strong extremal measure would be inhomogeneously extremal. The following proposition, proved in Section \ref{section:lower_bound}, answers this question.
\begin{proposition}\label{prop:inhomo_st-ext_imply_inhomo_ext}
    Suppose that $\mu$ be a measure supported on a subset of $F^{m \times n}.$ Then we have the following
    \begin{equation*}
        \text{$\mu$ is inhomogeneously strongly extremal } \implies \text{$\mu$ is inhomogeneously extremal}.
    \end{equation*}
\end{proposition}
  As mentioned earlier, in the homogeneous case, both the simultaneous and dual forms of Diophantine approximation lead us to the same notion of extremality due to Khintchine's transference principle. As there is no Khitchine's transference principle in the inhomogeneous case, one has to deal with simultaneous and dual forms of Diophantine approximation separately.
\begin{definition}
    A measure $\mu$ supported on a subset of $F^n$ is said to be \emph{dually inhomogeneously (strongly) extremal} if $\mu$ is inhomogeneously (strongly) extremal on $F^{1 \times n}$. We call $\mu$ \emph{simulteneously inhomogeneously (strongly) extremal} if $\mu$ is inhomogeneously (strongly) extremal on $F^{n \times 1}.$ Furthermore we say that $\mu$ is \emph{inhomogeneously (strongly) extremal} if $\mu$ is both dually and simultaneously inhomogeneously (strongly) extremal.
    
\end{definition}
In \cite{GG-Sprind}, Ganguly and Ghosh proved the inhomogeneous Sprind\v{z}uk conjecture over a field of positive characteristics.
\begin{theorem}[Theorem $1.1$, \cite{GG-Sprind}]\label{thm:GG}
     Let $\mathcal{M}$ be an analytic nondegenerate submanifold of $F^n.$ Then $\mathcal M$ is inhomogeneously extremal.
\end{theorem}

\subsection*{Acknowledgements}  We are indebted to the Department of Mathematics and Statistics, Indian Institute of Technology Kanpur, for providing the research friendly atmosphere. The first named author is also thankful to the International Centre  for Theoretical Sciences (ICTS), as a significant part of this project was completed during a visit (ICTS/etds-2022/12) to the same. Finally, we thank the anonymous referee for careful reading of this paper and helpful suggestions, which have improved the presentation of this paper.

\section{Main results}
\noindent In this section, we describe the main  results of this paper.
\begin{theorem}\label{thm:main}
\begin{enumerate}[label=(\Alph*)] 
    \item\label{thm:main_A} Let $\mu$ be an almost everywhere contracting measure on $F^{m \times n}$. Then
     \begin{equation*}
        \text{$\mu$ is  extremal} \iff \text{$\mu$ is inhomogeneously  extremal}.
    \end{equation*}
    \item \label{thm:main_B}  Let $\mu$ be an almost everywhere strongly contracting measure on $F^{m \times n}$. Then
     \begin{equation*}
        \text{$\mu$ is strongly extremal} \iff \text{$\mu$ is inhomogeneously strongly extremal}.
    \end{equation*}
\end{enumerate}

\end{theorem}
\noindent We will define \emph{(strongly) contracting measure} in the next section. Strongly contracting measures form a class of measures containing friendly measures. Hence in view of Theorem \ref{thm:main} and Corollary \ref{cor:fri_imply_contr}, we have the following inhomogeneous version of Theorem \ref{thm:KLW_fun}.
\begin{theorem}\label{thm:fri_inhomo_str_ext}
    Let $\mu$ be a friendly measure on $F^n$. Then $\mu $ is inhomogeneously strongly extremal.
\end{theorem}
\noindent It is well known that the natural measure supported on an analytic nondegenerate manifold is friendly. Hence as an immediate consequence of Theorem \ref{thm:main}, Theorem \ref{thm:G} and Corollary \ref{cor:fri_imply_contr}, we get the inhomogeneous Baker-Sprind\v{z}uk conjecture.

\begin{theorem}
Let $\mathcal{M}$ be an analytic nondegenerate submanifold of $F^n.$ Then $\mathcal M$ is inhomogeneously strongly extremal.    
\end{theorem}

\noindent Furthermore we prove the upper bounds for the non-extremal case in the following sense.
\begin{theorem}\label{thm:nonextremal_upper_bound}
\begin{enumerate}[label=(\Roman*)]
\item \label{thm:nonextremal_upper_bound_i}  Let $\mu$ be a measure on $F^{m \times n},$ which is contracting almost everywhere. Suppose that $\omega(Y)=\eta$ for $\mu$-almost every $Y \in F^{m \times n}.$ Then for all $\boldsymbol{\theta} \in F^m$ we have 

$$\omega(Y,\boldsymbol{\theta}) \leq \eta \,\, \,\, \text{for $\mu$-almost every $Y \in F^{m \times n}$}.$$
\item \label{thm:nonextremal_upper_bound_ii}  Let $\mu$ be a measure on $F^{m \times n},$ which is strongly contracting almost everywhere. Suppose that $\omega^{\times}(Y)=\eta$ for $\mu$-almost every $Y \in F^{m \times n}.$ Then for all $\boldsymbol{\theta} \in F^m$ we have 

$$\omega^{\times}(Y,\boldsymbol{\theta}) \leq \eta \,\, \,\, \text{for $\mu$-almost every $Y \in F^{m \times n}$}.$$
    \end{enumerate}
\end{theorem}
\noindent The above theorem is the function field analogue of \cite[Theorem 2.2]{GM}.

\section{Strongly contracting and friendly measures}\label{section:cont_meas}

 We retain the notations and terminologies of Beresnevich and Velani from \cite{BV}. Let $X$ be a metric space and $B \subseteq X$ be a ball. For $a>0,$ $aB$ denotes the ball with the same centre as $B$ and radius $a$ times the radius of $B$. We say that a measure $\mu$ on $X$ is \emph{non-atomic} if $\mu(\{x\})=0$ for any $x \in X.$ The \emph{support} of a measure is defined to be the smallest closed set $S$ such that $\mu(X \setminus S)=0.$ Also say that $\mu$ is \emph{doubling} if there exists a constant $c >0$ such that for any ball $B$ with centre in $S$
 \begin{equation*}
     \mu(2B) \leq c \mu(B).
 \end{equation*}
Consider the plane 
\begin{equation}
    \mathcal{L}_{\mathbf{b},\mathbf{c}} := \{Y \in F^{m \times n}: Y \mathbf{b} + \mathbf{c}=0\}
\end{equation}
for $\mathbf{b} \in F^n$ with $\|\mathbf{b}\|=1$ and $\mathbf{c} \in F^m.$ Now for $\bm{\varepsilon}=(\varepsilon_1,\dots,\varepsilon_m) \in (0,\infty)^m$, we define the $\bm{\varepsilon}$-neighbourhood of the plane $\mathcal{L}_{\mathbf{b},\mathbf{c}} $ by
\begin{equation}\label{equ:cont_plane}
    \mathcal{L}_{\mathbf{b},\mathbf{c}}^{(\bm{\varepsilon})} := \{Y \in F^{m \times n}: |Y_i \mathbf{b} + c_i|< \varepsilon_i \,\, \forall \,\, i=1,\dots,m\},
\end{equation}
where $Y_i$ is the $i$-th row of $Y.$ If $\varepsilon_1=\dots=\varepsilon_m=\varepsilon, $ we simply denote it by $ \mathcal{L}_{\mathbf{b},\mathbf{c}}^{(\varepsilon)}.$
\begin{definition}
    Let $\mu$ be a finite, non-atomic and doubling Borel measure on $F^{m \times n}.$ Then $\mu$ is said to be \emph{strongly contracting} if there exists $C,\alpha,r_0>0$ such that for any plane $\mathcal{L}_{\mathbf{b},\mathbf{c}},$ any $\bm{ \varepsilon } =(\varepsilon_1,\dots,\varepsilon_m) \in (0,\infty)^m$ with $\min \{\varepsilon_j : 1 \leq j \leq m\} < r_0$ and any $0<\delta<1,$ the following property is satisfied: $\forall \,\, Y \in  \mathcal{L}_{\mathbf{b},\mathbf{c}}^{(\delta \bm{\varepsilon})} \cap S$ there is an open ball $B$ centred at $Y$ such that
    \begin{equation}\label{equ:cm_pi}
        B \cap S \subset  \mathcal{L}_{\mathbf{b},\mathbf{c}}^{(\bm{\varepsilon})}
    \end{equation}
    and
    \begin{equation}\label{equ:cm_pii}
        \mu(5B \cap  \mathcal{L}_{\mathbf{b},\mathbf{c}}^{(\delta \bm{\varepsilon})}) \leq C \delta^{\alpha} \mu(5B ).
    \end{equation}
  We say that the measure $\mu$ is contracting if the same property holds with $\varepsilon_1 = \dots = \varepsilon_m =\varepsilon.$  
\end{definition}
\noindent Furthermore we say that $\mu$ is (strongly) contracting almost everywhere if for $\mu$-almost every point $Y_0 \in F^{m \times n}$ there is a neighbourhood $V$ of $Y_0$ such that the restriction $\mu|_V$ of $\mu$ to $V$ is (strongly) contracting. 

First of all, note that a measure $\mu$ is strongly contracting implies that $\mu$ is contracting.
Next, we recall the notion of friendly measures from \cite{KLW} and \cite{KT}. For this, we first need to define nonplanarity and $(C,\alpha)$-decaying of a measure $\mu.$  Let $\mu$ be a Borel merasure on $F^n$ and $S:= \text{the support of $\mu$}.$ We say $\mu$ is non-planar if for any hyperplane $\mathcal L,$ one has $\mu(\mathcal L )=0.$ Given a hyperplane $\mathcal L$ and a ball $B$ of $F^n$ with $\mu(B) >0,$ we define $\|d_{\mathcal{L}}\|_{\mu,B}$ to be the supremum of $\text{dist}(\mathbf{y},\mathcal{L})$ over $\mathbf{y} \in S \cap B,$ where $\text{dist}(\mathbf{y},\mathcal{L})= \inf \{\|\mathbf{y}-\mathbf{l}\|: \mathbf{l} \in \mathcal{L}\} .$ Given an open set V of $F^n$ and $C,\alpha>0$, we say that the measure $\mu$ is $(C,\alpha)$-decaying on $V$ if for any non-empty open ball $B \subset V$ centred in $S,$ any affine hyperplane $\mathcal{L}$ of $F^n$ and any $\varepsilon >0$ one has that
\begin{equation}
    \mu(B \cap \mathcal{L}^{(\varepsilon)}) \leq C \left(\frac{\varepsilon}{\|d_{\mathcal{L}}\|_{\mu,B}}\right)^{\alpha} \mu(B),
\end{equation}
where $\mathcal L^{(\varepsilon)}$ is the $\varepsilon$-neighbourhood of $\mathcal L$
\begin{definition}
   Let $\mu$ be a non-atomic Borel measure on $F^n.$ We say that the measure $\mu$ is friendly  if for $\mu$-almost  every point  $\mathbf{y}_0 \in F^n$, there exists a neighbourhood $V$ of $\mathbf{y}_0$ such that the restriction $\mu|_V$ of $\mu$ to $V$ is finite, doubling, non-planar and $(C,\alpha)$-decaying for some $C,\alpha>0.$ 
\end{definition}
Before going further let us introduce the notion of  $d$-contracting measure on $F^n$ following \cite{BV}, which generalizes the notion of contracting measure on $F^n$ (identified with $F^{1 \times n}$ or $F^{n \times 1}$).
\begin{definition}
    Let $d \in \mathbb N$ with $0 \leq d \leq n-1.$ A finite, non-atomic and doubling Borel measure on $F^n$ is said to be $d$-contracting if it satisfies the conditions \eqref{equ:cm_pi} and \eqref{equ:cm_pii} of contracting measure with the plane $\mathcal{L}_{\mathbf{b},\mathbf{c}}$ replaced by any $d$-dimensional plane $\mathcal{L}.$
\end{definition}
It is easy to note that a  contracting measure on $F^{1 \times n}$ is essentially same as a $0$-contracting measure on $F^n$  and  a contracting measure on $F^{n \times 1}$ is essentially same as a $(n-1)$-contracting measure on $F^n.$ We show that any friendly measure on $F^n$ is $d$-contracting for any $ 0 \leq d \leq n-1.$
\begin{theorem}\label{thm:friend_d_contr}
    Let $d \in \mathbb N $ with $0 \leq d \leq n-1.$ Then any friendly measure on $F^n$ is $d$-contracting almost everywhere.
\end{theorem}
\begin{proof}
Let $d \in \mathbb N$, $0 \leq d \leq n-1 $ and $\mu$ be a friendly measure on $F^{n}.$ Then for $\mu$-almost every point $\mathbf{y}_0 \in F^n$ there exists a neighbourhood $V$ of $\mathbf{y}_0$ such that the restriction $\mu|_V$ of $\mu$ to $V$ is non-planar, finite, doubling and $(C,\alpha)$-decaying on $V$ for some $C,\alpha >0.$ For the sake of simplicity, without loss of generality we may assume that $\mu=\mu|_V.$ We need to show that  there exists $C',\alpha',r_0>0$ such that for any $d$ dimensional plane $\mathcal{L}$ of $F^n,$ any $0 <\varepsilon < r_0 $  and any $0<\delta<1$ the following property is satisfied: $\forall \,\, \mathbf{y} \in  \mathcal{L}^{(\delta {\varepsilon})} \cap S$ there is an open ball $B$ centred at $\mathbf{y}$ such that
    \begin{equation}\label{equ:i}
        B \cap S \subset  \mathcal{L}^{({\varepsilon})}
    \end{equation}
    and
    \begin{equation}\label{equ:ii}
        \mu(5B \cap  \mathcal{L}^{(\delta {\varepsilon})}) \leq C' \delta^{\alpha'} \mu(5B ), 
    \end{equation}
where $S$ is the support of $\mu.$ 

\noindent First we observe that any $d$-dimensional plane of $F^n$ is of the form
\begin{equation}
    \mathcal L =\{\mathbf y=(y_1,\dots,y_n) \in F^n : b_{i1}y_1 + \dots+b_{in} y_n + a_i=0 \,\, \text{for} \,\, i=1,\dots,s\},
\end{equation}
where $\mathbf b_i=(b_{i1},\dots,b_{in}) \in F^n$ with $\|\mathbf b_i\|=1$ and $s \in \{1,\dots,n\}.$ 

\noindent Then
\begin{equation}
    \mathcal L^{(\varepsilon)} =\{\mathbf y=(y_1,\dots,y_n) \in F^n : |b_{i1}y_1 + \dots+b_{in} y_n + a_i|< \varepsilon \,\, \text{for} \,\, i=1,\dots,s\}.
\end{equation}

Since $\mu$ is nonplanar, the support $S$ of $\mu$ contains $n$ linear independent points $\mathbf{y}_1,\dots, \mathbf{y}_n$ of $F^n.$ Hence we can find a $r_0>0$ such that whenever $0< \varepsilon < r_0$, the $\varepsilon$-neighbourhood of any $d$-dimensional plane $\mathcal L$ cannot contain all the points $\mathbf{y}_1,\dots, \mathbf{y}_n.$ Therefore for any $d$-dimensional plane $\mathcal L$ and $0 < \varepsilon < r_0,$ we have that 
\begin{equation}\label{equ:iii}
    S \not\subset \mathcal{L}^{(\varepsilon)}.
\end{equation}  
Fix $0<\delta <1$. If $S \cap \mathcal{L}^{(\delta {\varepsilon})} = \emptyset,$ then there is nothing to show. So we now assume that   $S \cap \mathcal{L}^{(\delta \bm{\varepsilon})} \neq \emptyset$ and let $\mathbf{y} \in S \cap \mathcal{L}^{(\delta {\varepsilon})} .$  Our goal is to find a ball $B$ centred at $\mathbf{y}$ such that \eqref{equ:i} and \eqref{equ:ii} holds. 

\noindent Since $\mathcal{L}^{(\delta {\varepsilon})}$ is an open set, we can find a ball $B'$ centred at $\mathbf{y}$ such that
\begin{equation}\label{equ:iv}
    B' \subset \mathcal{L}^{(\varepsilon)}.
\end{equation}
In view of \eqref{equ:iii} and \eqref{equ:iv}, there exists a real number $\gamma \geq 1$ such that
\begin{equation}\label{equ:v}
5 \gamma B' \cap S \not\subset \mathcal{L}^{(\varepsilon)} \quad \text{and} \quad \gamma B' \cap S \subset \mathcal{L}^{(\varepsilon)}.
\end{equation}
Hence there exists a point $\mathbf{y}'=(y_1',\dots,y_n') \in (5 \gamma B' \cap S) \setminus \mathcal{L}^{(\varepsilon)} $ and this implies that
\begin{equation}\label{equ:plane_greater}
    |b_{l1}y_1' + \dots+b_{ln} y_n' + a_l| \geq  \varepsilon \,\, \text{for some} \,\, l \in \{1,\dots,s\}.
\end{equation}
Now we define the hyperplane
\begin{equation}
    \mathcal L_0 =\{\mathbf y=(y_1,\dots,y_n) \in F^n : b_{l1}y_1 + \dots+b_{ln} y_n + a_l=0 \}.
\end{equation}
Then in view of \eqref{equ:plane_greater} and using the ultrametric triangle inequality, we have $\text{dist}(\mathbf{y}',\mathcal{L}_0) \geq \varepsilon.$
 
Take $B=5\gamma B'.$ Then from above discussion it is clear that  $\|d_{\mathcal{L}_0}\|_{\mu,B} \geq \varepsilon .$ Now by applying the fact that $\mu$ is $(C,\alpha)$-decaying with $\delta \varepsilon$ in place of $\varepsilon,$ we get that
\begin{equation}
\mu(5\gamma B' \cap  \mathcal{L}^{(\delta {\varepsilon})}) \leq  \mu(5\gamma B' \cap  \mathcal{L}_{0}^{(\delta {\varepsilon})})  \leq C \left(\frac{\delta \varepsilon}{\varepsilon} \right)^{\alpha}  \mu(5 \gamma B')=   C \delta^{\alpha} \mu(5\gamma B' ).
\end{equation}
Therefore the ball $\gamma B'$ satisfies conditions \eqref{equ:i} and \eqref{equ:ii}. This implis that $\mu$ is $d$-contracting.
\end{proof}

As a natural consequence of the above theorem, we get the following:
\begin{corollary}\label{cor:fri_imply_contr}

\begin{enumerate}[label=(\arabic*)]
    \item Any friendly measure $\mu$ on $F^{m \times 1}$ is (strongly) contracting almost everywhere.
    \item Any friendly measure $\mu$ on $F^{1 \times n}$ is (strongly) contracting almost everywhere.
\end{enumerate}
\end{corollary}
\begin{proof}
    From Theorem \ref{thm:friend_d_contr}, it follows at once that any friendly measure on $F^{m \times 1}$ or $F^{1 \times n}$ is contracting almost everywhere.

    The proof of any friendly measure on $F^{m \times 1}$ or $F^{1 \times n}$ is strongly contracting almost everywhere, is exactly similar to the proof of Theorem \ref{thm:friend_d_contr}. Hence we skip the details to interested readers.
\end{proof}

\section{Lower bounds for Diophantine exponents}\label{section:lower_bound}
Part \ref{thm:main_B} of Theorem \ref{thm:main} is trivial and to prove the other derection, we need to show that certain measure $\mu$ on $F^{m \times n}$ is inhomogeneously strongly extremal. This amounts to showing that the following two statements hold for all $\boldsymbol{\theta} \in F^m$
\begin{equation}\label{equ:ub}
    \omega^{\times}(Y, \boldsymbol{\theta}) \leq 1 \text{ for $\mu$-almost every } Y \in F^{m \times n},
\end{equation}

and 
\begin{equation}\label{equ:lb}
    \omega^{\times}(Y, \boldsymbol{\theta}) \geq 1 \text{ for $\mu$-almost every }Y \in F^{m \times n}.
\end{equation}
Proving Part \ref{thm:main_A} of Theorem \ref{thm:main}, amounts to showing analogous upper and lower bounds for $\omega(Y,\boldsymbol{\theta}).$ We devote this section to prove \eqref{equ:lb}, i.e., the lower bound for inhomogeneous strong extremality. Also, we will see that as a by-product we get the lower bound for inhomogeneous extremality. Before going to the proof of this result, we recollect a few results related to the Diophantine transference principle over function fields. First, we recall the following result by  Bugeaud and Zhang (\cite{BZ}), which are the function field analogue of a result by Bugeaud and
Laurent (\cite{BL}).
\begin{theorem}\label{thm:bz}
    Let $Y \in F^{m \times n}$. Then for all $\boldsymbol{\theta} \in F^m,$ we have the following
    \begin{equation}\label{equ:bz_func}
        \omega(Y,\boldsymbol{\theta}) \geq \frac{1}{\hat{\omega}({Y}^t)} \text{ and } \quad \hat{\omega}(Y,\boldsymbol{\theta}) \geq \frac{1}{\omega(Y^t)}. 
    \end{equation}
   Furthermore, equality occurs in \eqref{equ:bz_func} for almost all $\boldsymbol{\theta} \in F^m$.
\end{theorem}
Next, we recall a positive characteristics version of Dyson's transference principle (\cite{D}) from \cite{GG-Sprind}.
\begin{theorem}\label{thm:dyson}
For any $Y \in F^{m \times n},$ we have the following
\begin{equation}
    \omega(Y) =1 \iff \omega(Y^t)=1.
\end{equation}
\end{theorem}
Now we are ready to prove the desired lower bound.
\begin{proposition}\label{prop:ext_w_W}
Let $\mu$ be an extremal measure on $F^{m \times n}.$ Then for every $\boldsymbol{\theta} \in F^m,$
\begin{equation*}
    \omega^{\times}(Y,\boldsymbol{\theta}) \geq \omega(Y,\boldsymbol{\theta}) \geq 1
\end{equation*}
for $\mu$-almost every $Y \in F^{m \times n}.$
\end{proposition}
\begin{proof}
It is clear from the definitions that
\begin{equation}\label{equ:wt}
 \omega^{\times}(Y,\boldsymbol{\theta}) \geq \omega(Y,\boldsymbol{\theta}) 
\end{equation}
for all $Y \in F^{m \times n}$ and $\boldsymbol{\theta} \in F^m.$ Also from Dirichlet's theorem and the definition of uniform exponent we have
\begin{equation}\label{equ:hw}
\hat{\omega}(Y) \geq 1 \text{ and } \omega(Y) \geq \hat{\omega} (Y)
\end{equation}
for all $Y \in F^{m \times n}.$
Since the given measure $\mu$ is extremal, $\omega(Y)=1$ for $\mu$-almost every $Y \in F^{m \times n}.$ Now applying Theorem \ref{thm:dyson}, we get that $\omega(Y^t)=1$ for $\mu$-almost every $Y \in F^{m \times n}$. Again from \eqref{equ:hw}, we get $\hat{\omega}(Y^t)=1$ for $\mu$-almost every $Y \in F^{m \times n}$. Finally using Theorem \ref{thm:bz} we get that 
\begin{equation}\label{equ:bz}
        \omega(Y,\boldsymbol{\theta}) \geq \frac{1}{\hat{\omega}({Y}^t)} \geq 1
    \end{equation}
    for $\mu$-almost every $Y \in F^{m \times n}.$ This completes the proof of our desired result in view of equation \eqref{equ:wt}.
\end{proof}

\subsection{Proof of Proposition \ref{prop:inhomo_st-ext_imply_inhomo_ext}} Suppose $\mu$ be an inhomogeneously strongly extremal measure on $F^{m \times n}.$  Then given any $\boldsymbol{\theta} \in F^m,$ $\omega^{\times}(Y,\boldsymbol{\theta})=1$ for $\mu$-almost every $Y \in F^{m \times n}.$ Therefore $\omega(Y,\boldsymbol{\theta}) \leq 1$ for all $\boldsymbol{\theta} \in F^m$ and for $\mu$-almost every $Y \in F^{m \times n},$ by using  \eqref{equ:mr1}. \\

\noindent We will be done if we can show that  $\omega(Y,\boldsymbol{\theta}) \geq 1$ for all $\boldsymbol{\theta} \in F^m$ and for $\mu$-almost every $Y \in F^{m \times n}.$ Since $\mu$ is inhomogeneously strongly extremal, we trivially have that $\mu$ is extremal. By applying Proposition \ref{prop:ext_w_W}, we get our desired result.

\section{Inhomogeneous Transference Principle}\label{section:itp} The main ingredient in the proof of Theorem \ref{thm:main} and Theorem \ref{thm:nonextremal_upper_bound}, is the inhomogeneous transference principle \cite[Section 5]{BV}. Suppose that $(X,d)$ be a locally compact metric space. Given two countable indexing set $\mathcal{A}$ and $\mathbf{T},$ let $H$ and $I$ be two maps from $\mathbf{T} \times \mathcal{A} \times \mathbb{R}_+$ into the set of open subsets of $X$ such that
\begin{equation}\label{equ:H}
    H \,\, : \,\,  (\mathbf{t},\alpha, \lambda) \in \mathbf{T} \times \mathcal{A} \times \mathbb{R}_+ \mapsto H_{\mathbf{t}}(\alpha,\lambda)
\end{equation}
and
\begin{equation}\label{equ:I}
     I \,\, : \,\, (\mathbf{t},\alpha, \lambda) \in \mathbf{T} \times \mathcal{A} \times \mathbb{R}_+ \mapsto I_{\mathbf{t}}(\alpha,\lambda).
\end{equation}
Also, let
\begin{equation}\label{equ:H_t_I_t}
    H_{\mathbf{t}}(\lambda)\defeq \bigcup_{\alpha \in \mathcal{A}} H_{\mathbf{t}}(\alpha,\lambda) \quad \text{and} \quad I_{\mathbf{t}}(\lambda)\defeq  \bigcup_{\alpha \in \mathcal{A}} I_{\mathbf{t}}(\alpha,\lambda).
\end{equation}
Let $\Phi$ denote a set of functions $\phi : \mathbf{T} \rightarrow \mathbb{R}_+.$ For $\phi \in \Phi,$ consider the limsup sets
\begin{equation}\label{equ:lambda_H_I}
    \Lambda_H(\phi) = \displaystyle \limsup_{\mathbf{t} \in \mathbf{T}} H_{\mathbf{t}}(\phi(\mathbf{t})) \quad \text{and} \quad \Lambda_I(\phi) = \displaystyle \limsup_{\mathbf{t} \in \mathbf{T}} I_{\mathbf{t}}(\phi({\mathbf{t}})).
\end{equation}
We call the sets associated with the maps $H$ and $I$ as homogeneous and inhomogeneous sets respectively. Now we discuss two important properties, which are the key ingredients for the inhomogeneous transference principle. \\

\noindent$\textbf{The intersection property.} $ We say that the triple $(H,I,\Phi)$  satisfy the \emph{intersection property} if for any $\phi \in \Phi,$ there exists $\phi^* \in \Phi$ such that for all but finitely many $\mathbf{t} \in T$ and all distinct $\alpha$ and $\alpha'$ in $\mathcal{A},$ we have that
\begin{equation}\label{equ:inter_property}
    I_{\mathbf{t}}(\alpha,\phi(\mathbf{t})) \cap I_{\mathbf{t}}(\alpha',\phi(\mathbf{t})) \subseteq H_{\mathbf{t}}(\phi^*(\mathbf{t})).
\end{equation}

\noindent $\textbf{The contraction property.}$ We suppose that $\mu$ be a non-atomic finite doubling measure supported on a bounded subset $\mathbf{S}$ of $X.$

\noindent  $\mu$ is said to be $\emph{contracting with respect}$ to $(I, \Phi)$ if for any $\phi \in \Phi,$ there exists $\phi^+ \in \Phi$ and a sequence of positive numbers $\{ k_{\mathbf{t}} \}_{\mathbf{t} \in \mathbf{T}}$ satisfying
\begin{equation}\label{equ:itp_conv}
    \displaystyle \sum_{\mathbf{t} \in \mathbf{T}} k_{\mathbf{t}} < \infty,
\end{equation}
and for all but finitely many $\mathbf{t} \in \mathbf{T}$ and all $\alpha \in \mathcal{A}$ there exists a collection $C_{\mathbf{t},\alpha}$ of balls $B$ centred in $\mathbf{S}$ satisfying the following three conditions:
\begin{enumerate}[label=(C.{{\arabic*}})]
    \item \label{equ:cont_prop_i} $\mathbf{S} \cap I_{\mathbf{t}}(\alpha,\phi({\mathbf{t}})) \subseteq \displaystyle \bigcup_{B \in C_{\mathbf{t},\alpha}} B$, 
    \item \label{equ:cont_prop_ii} 
    $\mathbf{S} \cap  \displaystyle \bigcup_{B \in C_{\mathbf{t},\alpha}} B \subseteq I_{\mathbf{t}}(\alpha,\phi^+(\mathbf{t}))$, and 
    \item \label{equ:cont_prop_iii}
    $\mu(5B \cap I_{\mathbf{t}}(\alpha,\phi(\mathbf{t}))) \leq k_{\mathbf{t}} \mu(5B)$.
\end{enumerate}
We now recall the following transference theorem from \cite[Theorem 5]{BV}.
\begin{theorem}\label{thm:transf}
Let the triple $(H,I,\Phi)$ satisfies the intersection property and $\mu$ is contracting with respect to $(I,\Phi).$ Then
\begin{equation}
    \forall \phi \in \Phi,\, \mu(\Lambda_H(\phi))=0 \Longrightarrow \forall \phi \in \Phi,\,\mu(\Lambda_I(\phi))=0.
\end{equation} 
\end{theorem}

\section{Upper bounds for Diophantine exponents}
 
In this section, we will prove  Part \ref{thm:main_B} of Theorem \ref{thm:main} (the proof of Part \ref{thm:main_A} of Theorem \ref{thm:main} is similar to that of Part \ref{thm:main_B}). Clearly, one direction is trivial. For the other direction, let $\mu$ be a measure on $F^{m \times n}$, which is strongly contracting almost everywhere and  we want to show that 
\begin{equation}\label{equ:ub1}
    \text{$\mu$ is strongly extremal} \quad \implies \quad \text{$\mu$ is inhomogeneously strongly extremal}.
\end{equation}
Earlier we observed that this amounts to showing the upper and lower bounds, \eqref{equ:ub} and \eqref{equ:lb} respectively. We have already proved the lower bound in the last section. We prove the upper bound in this section using the inhomogeneous transference principle. In fact, we will prove the upper bound for the general non-extremal case (i.e., Part \ref{thm:nonextremal_upper_bound_ii} of Theorem \ref{thm:nonextremal_upper_bound}). We follow the strategy of Beresnevich and Velani (\cite{BV}).

 Let $\mu$ be a measure on $F^{m \times n},$ which is strongly contracting almost everywhere. Suppose that $\omega^{\times}(Y)=\eta $ for $\mu$-almost every $Y \in F^{m \times n}.$ We want to prove that, $\forall \boldsymbol{\theta} \in F^m$, $\omega^{\times}(Y,\boldsymbol{\theta}) \leq \eta$ for $\mu$-almost every $Y \in F^{m \times n}.$
 
 \noindent Keeping this in our mind, we define
\begin{equation}\label{equ:ub2}
\mathcal{U}_{m,n}^{\boldsymbol{\theta}}(\eta) :=\{Y \in F^{m \times n} : \omega^{\times}(Y, \boldsymbol{\theta}) > \eta \}.
\end{equation}
Observe that Part \ref{thm:nonextremal_upper_bound_ii} of Theorem \ref{thm:nonextremal_upper_bound}  reduces to proving the following:
\begin{equation}\label{equ:ub3}
    \mu(\mathcal{U}_{m,n}^{\boldsymbol{\theta}}(\eta))=0 \quad \text{for all $\boldsymbol{\theta} \in F^m$}.
\end{equation}
Consider
\begin{equation}\label{equ:ub4}
    \mathbf{T}=\mathbb Z^{m+n}.
\end{equation}
For each $\mathbf{t}=(t_1,\dots,t_{m+n})\in \mathbf{T},$ let
\begin{equation}\label{equ:ub5}
a_{\mathbf{t}}:=\diag \{X^{t_1},\dots,X^{t_m},X^{-t_{m+1}},\dots,X^{-t_{m+n}}\} \in F^{m \times n}.
\end{equation}
For any matrix $Y \in F^{m \times n},$ let
$$U_Y :=\begin{bmatrix}
                 I_m& Y\\0& I_n
\end{bmatrix},$$
where $I_k$ denotes the identity matrix of size $k \times k.$ We can view $U_{Y}$ as a linear operator on $F^{m + n}.$ Now given any $\mathbf{\boldsymbol{\theta}}=(\theta_1,\dots,\theta_m) \in F^m,$ we define the following affine transformation $U_Y^{\mathbf{\boldsymbol{\theta}}}$ on $F^{m + n}$.
\begin{equation*}\label{equ:ub6}
    U_Y^{\mathbf{\boldsymbol{\theta}}}(\mathbf{a}):=U_Y^{\mathbf{\boldsymbol{\theta}}}\mathbf{a}:=U_Y (\mathbf{a})+ \Theta  \quad \forall \,\, \mathbf{a} \in F^{m+n},
\end{equation*}
 where $\Theta=({\theta}_1,\dots,{\theta}_m,0,\dots,0)^t \in F^{m+n}.$
 
\noindent Let 
\begin{equation}\label{equ:ub7}
\mathcal{A}= \Lambda^m \times (\Lambda^n \setminus \{0\}).
\end{equation}
For $\varepsilon>0,$ $\mathbf{t} \in T$ and $\alpha \in \mathcal{A},$ let
\begin{equation}\label{equ:ub8}
\Delta_{\mathbf{t}}^{\boldsymbol{\theta}}(\alpha,\varepsilon):=\{Y \in F^{m \times n}:\|a_{\mathbf t}U_{Y}^{\boldsymbol{\theta}} \alpha\| < \varepsilon \}
\end{equation}
and 
\begin{equation}\label{equ:ub9}
\Delta_{\mathbf{t}}^{\boldsymbol{\theta}}(\varepsilon):= \bigcup_{\alpha \in \mathcal A} \Delta_{\mathbf{t}}^{\boldsymbol{\theta}}(\alpha,\varepsilon)=\{Y \in F^{m \times n}: \inf_{\alpha \in \mathcal{A}} \|a_{\mathbf t}U_{Y}^{\boldsymbol{\theta}} \alpha \| < \varepsilon \}.
\end{equation}
For $\tau>0,$ consider the function
\begin{equation}\label{equ:ub10}
\phi^{\tau} : \mathbf{T} \rightarrow \mathbb R_{+} ;\,\, \mathbf{t} \mapsto \phi^{\tau}_{\mathbf{t}}:= e^{-\tau \sigma(\mathbf{t})}
\end{equation}
where $\sigma(\mathbf{t}):=t_1+\dots+t_{m+n},$ and the following set
\begin{equation}\label{equ:ub11}
\Delta_{\mathbf{T}}^{\boldsymbol{\theta}}(\phi^{\tau}):= \limsup_{\mathbf t \in \mathbf T} \Delta_{\mathbf{t}}^{\boldsymbol{\theta}}(\phi^{\tau}).
\end{equation}
For $\boldsymbol{\theta}=\boldsymbol{0} ,$ i.e., the homogeneous case, we denote $\Delta_{\mathbf{T}}^{\boldsymbol{\theta}}(\phi^{\tau})$ by $\Delta_{\mathbf{T}}(\phi^{\tau}).$ The following proposition allows us to reformulate the set $\mathcal{U}_{m,n}^{\boldsymbol{\theta}}(\eta)$ in terms of the above lim sup sets.

\begin{proposition}\label{prop:main}
    There exists a subset $\mathbf{T}$ of $ \mathbb Z^{m + n}$ such that
    \begin{equation}\label{equ:ub12}
        \sum_{\mathbf t \in \mathbf T} e^{-\tau \sigma(\mathbf t)} < \infty \quad \forall \,\, \tau>0
    \end{equation}
    and
    \begin{equation}\label{equ:ub13}
        \mathcal{U}_{m,n}^{\boldsymbol{\theta}}(\eta)=\bigcup_{\tau>0} \Delta_{\mathbf{T}}^{\boldsymbol{\theta}}(\phi^{\tau}), \quad \forall \,\, \boldsymbol{\theta} \in F^m.
    \end{equation}
\end{proposition}
\begin{proof}
First, we define the set $\mathbf T \subset \mathbb Z^{m+n}$ which will do the job for us. For $\mathbf{u}=(u_1,\dots,u_m) \in \mathbb Z^{m}_{+}$ and $\mathbf{v}=(v_1,\dots,v_n) \in \mathbb Z^{n}_{+} ,$ we define the following
\begin{equation*}\label{equ:p1}
    \sigma(\mathbf u):= \sum_{j=1}^{m} u_j, \quad  \sigma(\mathbf v):= \sum_{i=1}^{n} v_i, \text{ and } \xi:=\xi(\mathbf u,\mathbf v)=\frac{\sigma(\mathbf u )- \eta \sigma(\mathbf v)}{m+\eta n},
\end{equation*}
where $\mathbb Z_{+}=\{s \in \mathbb Z: s \geq 0\}.$ Given $\mathbf u$ and $\mathbf v$ as above, define the $(m+n)$-tuple $\mathbf{t}=(t_1,\dots,t_{m+n})$ as
\begin{equation}\label{equ:p2}
    \mathbf{t}:=(u_1-[\xi],\dots,u_m - [\xi], v_1 + [\xi] , \dots, v_n +[\xi]),
\end{equation}
where for any $x \in \mathbb R,$ $[x]$ denotes the greatest integer not greater than $x.$ Finally we define
\begin{equation}\label{equ:p3}
    \mathbf T := \{\mathbf t \in \mathbb Z^{m+n} \,\, \text{given by} \,\, \eqref{equ:p2}: \mathbf u \in \mathbb Z^{m}_{+}, \,\, \mathbf{v} \in \mathbb Z^{n}_{+} \,\, \text{with} \,\, \sigma(\mathbf u) \geq \eta \sigma(\mathbf v)\}  .
\end{equation}
We show that the above defined $\mathbf T$ will solve the purpose of Proposition \ref{prop:main}. First, let us record a few inequalities which will be essential later.
\begin{align}\label{equ:p4}
 \sigma(\mathbf t)  :=  \sum_{i=1}^{m+n} t_i  &=  \sigma(\mathbf u) - m [\xi]+\sigma(\mathbf v)+ n[\xi]  \nonumber  \\ & =  \frac{\eta +1}{\eta} \sigma(\mathbf u)-m\left(\frac{\xi}{\eta} + [\xi] \right)+n([\xi]- \xi),
\end{align}
and also
\begin{equation}\label{equ:p5}
    \sigma(\mathbf t)= (\eta+1) \sigma(\mathbf v) +m(\xi -[\xi])+n([\xi]+\eta \xi ).
\end{equation}
From \eqref{equ:p4} and \eqref{equ:p5}, we respectively get
\begin{align}\label{equ:p6}
  \frac{\eta}{\eta+1} \sigma(\mathbf t)   & =  \sigma(\mathbf u)- \frac{m \eta}{\eta +1}\left(\frac{\xi}{\eta} + [\xi] \right)+\frac{n \eta}{\eta +1}([\xi]- \xi) \nonumber \\ & \leq \sigma(\mathbf u)- \frac{m \eta}{\eta +1}\left(\frac{\xi}{\eta}+ [\xi] \right), \,\, \text{since} \,\, \xi-1 < [\xi] \leq \xi \nonumber \\ & = \sigma(\mathbf u)- \frac{m \eta}{\eta +1}\left(\frac{[\xi]}{\eta}+ [\xi] \right) + \frac{m \eta}{\eta +1} \left(\frac{[\xi]}{\eta}-\frac{\xi}{\eta} \right) \nonumber \\ & \leq \sigma(\mathbf{u}) - m [\xi], \,\, \text{since} \,\, \xi-1 < [\xi] \leq \xi
\end{align}
and 
\begin{align}\label{equ:p7}
\frac{1}{\eta+1}\sigma(\mathbf t) &=  \sigma(\mathbf v) +\frac{m}{\eta+1}(\xi -[\xi])+\frac{n}{\eta+1}([\xi]+\eta \xi ) \nonumber \\ & \geq \sigma(\mathbf v) +\frac{n}{\eta+1}([\xi]+\eta \xi ), \,\, \text{since} \,\, \xi-1 < [\xi] \leq \xi \nonumber \\ &= \sigma(\mathbf v) +\frac{n}{\eta+1}([\xi]+\eta [\xi] )+\frac{n}{\eta+1} \left(\eta \xi-\eta[\xi]\right) \nonumber \\ & \geq  \sigma(\mathbf{v}) + n [\xi], \,\, \text{as} \,\, \xi-1 < [\xi] \leq \xi.
\end{align}
Since $\xi \geq 0,$ in view of \eqref{equ:p6} and \eqref{equ:p7} we have
\begin{equation}\label{equ:p6i}
    (\eta+1)\sigma(\mathbf v) \leq  \sigma(\mathbf t) \leq \frac{\eta+1}{\eta}\sigma(\mathbf u).
\end{equation}
To show that the series \eqref{equ:ub12} is convergent for our choice of $\mathbf T,$ we give the following estimate.

\begin{align}\label{equ:p7i}
\sigma(\mathbf t)= \frac{\eta}{\eta+1} \sigma(\mathbf{t})+ \frac{1}{\eta+1} \sigma(\mathbf{t}) & =\sigma(\mathbf{u}) + \sigma(\mathbf{v}) - (m-n) [\xi] \nonumber \\ & \geq \sigma(\mathbf{u}) + \sigma(\mathbf{v}) - (m-n) \xi \nonumber \\ & = \sigma(\mathbf{u}) + \sigma(\mathbf{v}) - \frac{(m-n)}{m+\eta n} (\sigma(\mathbf{u}) - \eta \sigma(\mathbf{v}) ) \nonumber \\ & = \frac{\eta +1}{m+\eta n}(n \sigma(\mathbf{u}) + m \sigma(\mathbf{v})),
\end{align}
by using \eqref{equ:p6}, \eqref{equ:p7} and $\xi-1 < [\xi] \leq \xi.  $
In view of \eqref{equ:p7i}, we conclude that the series \eqref{equ:ub12} is convergent and furthermore we get that for any $r \in \mathbb R_+$
\begin{equation}\label{equ:p8i}
  \# \{ \mathbf t \in \mathbf T : \sigma(\mathbf t) < r\}  < \infty.
\end{equation}
First, we show that 
\begin{equation}\label{equ:p8}
\mathcal{U}_{m,n}^{\boldsymbol{\theta}}(\eta) \subseteq \bigcup_{\tau>0} \Delta_{\mathbf T}^{\boldsymbol{\theta}}(\phi^{\tau}).
\end{equation}
Let $Y \in \mathcal{U}_{m,n}^{\boldsymbol{\theta}}(\eta). $ Note that $Y \in \mathcal{U}_{m,n}^{\boldsymbol{\theta}}(\eta)$ if and only if $\exists$ $\varepsilon>0,$ such that for arbitrarily large $T  \geq 1$ there exists $\alpha =(\mathbf{p}, \mathbf q) \in \mathcal A=\Lambda^m \times (\Lambda^n \setminus \{0\})$ satisfying $\|Y \mathbf q + \mathbf p+ \boldsymbol{\theta}\| \leq 1/e$ such that
\begin{equation}\label{equ:p9}
\prod (Y \mathbf q +\mathbf p + \boldsymbol{\theta}) < e^{-(\eta+\varepsilon)T}  \text{ and } {\prod}_+ (\mathbf q) < e^T.
\end{equation}
Since $Y \in \mathcal{U}_{m,n}^{\boldsymbol{\theta}}(\eta),$ \eqref{equ:p9} is satisfied for infinitely many $T\in \mathbb N$. Thus for any such $T ,$ $\exists$ unique $\mathbf{u}=(u_1,\dots,u_m) \in \mathbb Z^{m}_{+}$ and $\mathbf{v}=(v_1,\dots,v_n) \in \mathbb Z^{n}_{+}$ such that
\begin{equation}\label{equ:p10}
e^{-u_j} \leq \max \left\{ |Y_j \mathbf q +p_j +{\theta}_j|, e^{-(\eta+\varepsilon)T}  \right\} < e^{-u_j+1} \quad \text{for} \quad 1 \leq j \leq m
\end{equation}
and
\begin{equation}\label{equ:p11}
e^{v_i} \leq \max \{1,|q_i|\} < e^{v_i +1} \quad \text{for} \quad 1 \leq i \leq n,
\end{equation}
where $Y_j$ denotes the $j$-th row of $Y \in F^{m \times n}.$

\noindent Using \eqref{equ:p10} and \eqref{equ:p11}, we get that
\begin{equation}\label{equ:p12}
e^{-\sigma(\mathbf u)} < \max \left\{ \prod(Y \mathbf q + \mathbf p + \boldsymbol{\theta}), e^{-(\eta+\varepsilon)T} \right\} \quad \text{and} \quad e^{\sigma(\mathbf v)} \leq {\prod}_+ (\mathbf q).
\end{equation}
Now \eqref{equ:p9} and \eqref{equ:p12} implies that $e^{-\sigma(\mathbf u)} < e^{-\sigma(\mathbf v)(\eta+\varepsilon)}.$ Therefore
\begin{equation}\label{equ:p13}
    \sigma(\mathbf u) - \eta \sigma(\mathbf v) > \varepsilon \sigma(\mathbf v) \geq 0.
\end{equation}
Hence $\mathbf t$ given by \eqref{equ:p2} with $\mathbf u=(u_1,\dots,u_m)$ and $\mathbf v=(v_1,\dots,v_n)$ satisfying \eqref{equ:p10} and \eqref{equ:p11} respectively, is in $\mathbf T. $

%\noindent Note that 
%\begin{equation}\label{equ:p14}
    %a_{\mathbf t}U_{Y}^{\boldsymbol{\theta}} \alpha=\left(X^{u_1}(Y_1 \mathbf q+p_1+\boldsymbol{\theta}_1),\dots,X^{u_m}(Y_m \mathbf q+p_m+\boldsymbol{\theta}_m),X^{-v_1}q_1,\dots,X^{-v_n}q_n\right).
%\end{equation}

\noindent If $\sigma(\mathbf u) \leq 2 \eta \sigma (\mathbf v),$ then using \eqref{equ:p6i} and \eqref{equ:p13} we get
\begin{equation*}\label{equ:p15}
    \xi=\frac{\sigma(\mathbf u )- \eta \sigma(\mathbf v)}{m+\eta n} \geq \frac{\varepsilon \sigma(\mathbf v)}{m+\eta n} \geq \frac{\varepsilon \sigma(\mathbf u)}{2 \eta (m+n)} > \frac{\varepsilon \sigma (\mathbf t)}{2(\eta +1)(m+n)}.
\end{equation*}
If $\sigma(\mathbf u) > 2 \eta \sigma (\mathbf v),$ then using \eqref{equ:p6i} we get
\begin{equation*}\label{equ:p16}
    \xi=\frac{\sigma(\mathbf u )- \eta \sigma(\mathbf v)}{m+ \eta n} = \frac{2\sigma(\mathbf u )- 2\eta \sigma(\mathbf v)}{2(m+ \eta n)} > \frac{ \sigma(\mathbf u)}{2(m+\eta n)} > \frac{\eta \sigma (\mathbf t)}{2(\eta +1)(m+ \eta n)}.
\end{equation*}
In view of the above two inequalities, we get that
\begin{equation}\label{equ:p17}
    \xi > \tau_0 \sigma(\mathbf t), \quad \text{where} \quad \tau_0:=\frac{ \min\{\varepsilon,\eta \}}{2(\eta +1)(m+\eta n)}.
\end{equation}
Now we take
\begin{equation}\label{equ:p18}
a_{\mathbf t}=X^{-([\xi]+1)} \diag \{X^{u_1},\dots,X^{u_m},X^{-v_1},\dots,X^{-v_n}\}.
\end{equation}
Then using \eqref{equ:p10} and \eqref{equ:p11},  we get
\begin{equation}\label{equ:p19}
\inf_{\alpha \in \mathcal A} \| a_{\mathbf t}U_{Y}^{\boldsymbol{\theta}} \alpha\| < e \cdot e^{-([\xi]+1)} \leq e \cdot e^{- \xi}.
\end{equation}
Combining \eqref{equ:p17} and \eqref{equ:p19}, we conclude that, for $0<\tau<\tau_0$
\begin{equation}\label{equ:p20}
\inf_{\alpha \in \mathcal A} \| a_{\mathbf t}U_{Y}^{\boldsymbol{\theta}} \alpha\| <  e^{-\tau \sigma(\mathbf t)} 
\end{equation}
for all sufficiently large $\sigma(\mathbf t).$ \eqref{equ:p9} coupled with \eqref{equ:p10} implies that $\sigma(\mathbf u) \rightarrow \infty$ as $T \rightarrow \infty.$ Hence in view of  \eqref{equ:p7i} and the fact that \eqref{equ:p9} holds for arbitrarily large  $T \in \mathbb N,$ we conclude that  \eqref{equ:p20} holds for infinitely many $\mathbf t \in \mathbf T.$  Therefore $Y \in  \Delta_{\mathbf T}^{\boldsymbol{\theta}}(\phi^{\tau})$ for any $\tau \in (0,\tau_0).$ This completes the proof of \eqref{equ:p8}.\\

 Finally to complete the proof of Proposition \ref{prop:main}, we show that 
\begin{equation}\label{equ:p21}
\mathcal{U}_{m,n}^{\boldsymbol{\theta}}(\eta) \supseteq \bigcup_{\tau>0} \Delta_{\mathbf T}^{\boldsymbol{\theta}}(\phi^{\tau}).
\end{equation}
Let $Y \in \Delta_{\mathbf T}^{\boldsymbol{\theta}}(\phi^{\tau})$ for some $\tau>0. $ By definition
\begin{equation*}\label{equ:p22}
    \inf_{\alpha \in \mathcal A} \| a_{\mathbf t}U_{Y}^{\boldsymbol{\theta}} \alpha\| <  e^{-\tau \sigma(\mathbf t)} 
\end{equation*}
for infinitely many $\mathbf t \in \mathbf T.$ For any such $\mathbf t,$ $\exists$ $\alpha=(\mathbf p, \mathbf q) \in \mathcal A$ such that
\begin{equation*}\label{equ:p23}
    \| a_{\mathbf t}U_{Y}^{\boldsymbol{\theta}} \alpha\| <  e^{-\tau \sigma(\mathbf t)}. 
\end{equation*}
By taking the product of first $m$ coordinates and last $n$ non-zero coordinates of $a_{\mathbf t}U_{Y}^{\boldsymbol{\theta}} \alpha$ we respectively get
\begin{equation}\label{equ:p24}
\prod_{j=1}^{m} e^{t_j}|Y_j \mathbf q + p_j+{\theta}_j| < e^{-m \tau \sigma(\mathbf t)}
\end{equation}
and
\begin{equation}\label{equ:p25}
\prod_{1 \leq i \leq n, \,\, q_i \neq 0} e^{-t_{m+i}} |q_i| < e^{- n \tau \sigma(\mathbf t)}.
\end{equation}
Using \eqref{equ:p6i} and the fact that  $\sigma(\mathbf{t}) \geq 0$ (by \eqref{equ:p7}), we get 
\begin{equation}\label{equ:p26}
\prod (Y \mathbf q + \mathbf p +\boldsymbol{\theta} ) < e^{-m \tau \sigma(\mathbf t)} \cdot e^{-\frac{\eta}{\eta +1} \sigma(\mathbf t)}=e^{-\left(\eta +m\tau(\eta+1)\right)\frac{1}{\eta+1} \sigma(\mathbf t)}
\end{equation}
and
\begin{equation}\label{equ:p27}
{\prod}_+ (\mathbf q) < e^{- n \tau \sigma(\mathbf t)} e^{\frac{1}{\eta+1} \sigma(\mathbf t)} < e^{\frac{1}{\eta+1} \sigma(\mathbf t)}.
\end{equation}
 If we take $T=\frac{1}{\eta+1} \sigma(\mathbf t)$ and $\varepsilon :=m \tau (\eta +1),$ then \eqref{equ:p26} and \eqref{equ:p27} holds for arbitrarily large $T.$ Therefore $Y \in \mathcal{U}_{m,n}^{\boldsymbol{\theta}}(\eta)$ and this completes the proof of Proposition \ref{prop:main}.
\end{proof}

\subsection{Proof of Theorem \ref{thm:nonextremal_upper_bound}}
In view of Equation \eqref{equ:ub3} and Proposition \ref{prop:main}, we first note that in order to prove Part \ref{thm:nonextremal_upper_bound_ii} of Theorem \ref{thm:nonextremal_upper_bound} it is enough to show that
\begin{equation}\label{equ:ub14}
\mu(\Delta_{\mathbf{T}}(\phi^{\tau}))=0 \quad \forall \,\, \tau>0 \quad \implies \mu(\Delta_{\mathbf{T}}^{\boldsymbol{\theta}}(\phi^{\tau})) \quad \forall \,\, \tau>0. 
\end{equation}
We use the inhomogeneous transference principle to prove \eqref{equ:ub14}. From now on, $\boldsymbol{\theta} \in F^m$ is fixed and, without loss of generality, we assume that $\mu$ is a strongly contracting on $F^{m \times n}.$ Let $X:=F^{m \times n}$, $\mathcal A= \Lambda^m \times (\Lambda^n \setminus \{0\}),$ $\mathbf T$ be as in Proposition \ref{prop:main}, and the maps $H$ and $I$ are given by 
\begin{equation}\label{equ:ub15}
H_{\mathbf t}(\alpha,\varepsilon):= \Delta_{\mathbf{t}}(\alpha,\varepsilon)=\Delta_{\mathbf{t}}^{\boldsymbol{0}}(\alpha,\varepsilon) \text{ and }I_{\mathbf{t}}(\alpha,\varepsilon):= \Delta_{\mathbf{t}}^{\boldsymbol{\theta}}(\alpha,\varepsilon),
\end{equation}
  where $\varepsilon > 0 ,$ $\mathbf{t} \in \mathbf T$, $\alpha \in \mathcal A$ and $\Delta_{\mathbf{t}}^{\boldsymbol{\theta}}(\alpha,\varepsilon)$ is given by \eqref{equ:ub8}. In view of the above definitions, it now readily follows that $H_{\mathbf t}(\varepsilon)=\Delta_{\mathbf t}^{\boldsymbol{0}}(\varepsilon)$ and $I_{\mathbf{t}}(\varepsilon)=\Delta_{\mathbf{t}}^{\boldsymbol{\theta}} (\varepsilon).$ Let $\Phi$ be the collection of functions given by \eqref{equ:ub10}. Then we have
  \begin{equation}\label{equ:ub16}
\Lambda_H(\phi)=\Delta_{\mathbf T}(\phi):=\Delta_{\mathbf T}^{\boldsymbol{0}}(\phi) \text{ and } \Lambda_I(\phi)=\Delta_{\mathbf T}^{\boldsymbol{\theta}}(\phi),
  \end{equation}
where $\Delta_{\mathbf T}^{\boldsymbol{\theta}}(\phi)$ is given by \eqref{equ:ub11}. Now \eqref{equ:ub14} will immediately follow by applying the inhomogeneous transference principle if we can verify that the triple $(H,I,\Phi)$ satisfies the intersection property and the measure $\mu$ is contracting with respect to $(I,\Phi).$

\subsubsection{The intersection property}
Let $\alpha=(\mathbf{p},\mathbf{q})$ and $\alpha'=(\mathbf{p}',\mathbf{q}')$ be two distinct elements in $\mathcal{A},$ where $\mathbf{p},\mathbf{p}' \in \Lambda^m$ and $\mathbf{q},\mathbf{q}' \in \Lambda^n \setminus \{0\}$. Also let $\phi \in \Phi.$ Then $\phi(\mathbf t)=e^{-\tau \sigma(\mathbf t)}$ for some $\tau >0.$ Consider any element $Y$ from $I_{\mathbf{t}}(\alpha,\phi(\mathbf{t})) \cap I_{\mathbf{t}}(\alpha',\phi(\mathbf{t})) .$ Then we have
\begin{equation}\label{equ:ub17}
\|a_{\mathbf{t}}U_{Y}^{\boldsymbol{\theta}} \alpha \|<\phi(\mathbf{t}) \quad and \quad \|a_{\mathbf{t}}U_{Y}^{\boldsymbol{\theta}} \alpha' \|<\phi(\mathbf{t}).
\end{equation}
Since $\mathbf T$  satisfies \eqref{equ:ub12} and \eqref{equ:p8i}
\begin{equation}\label{equ:ub18}
    \|a_{\mathbf{t}}U_{Y} (\alpha-\alpha') \|= \|a_{\mathbf{t}}U_{Y}^{\boldsymbol{\theta}} \alpha-a_{\mathbf{t}}U_{Y}^{\boldsymbol{\theta}} \alpha' \|  <\phi(\mathbf t)
\end{equation}
for all but finitely many $\mathbf{t} \in \mathbf T.$ Let $\alpha'':=\alpha -\alpha'=(\mathbf{p}'',\mathbf{q}''),$ where $\mathbf{p}''=\mathbf{p}-\mathbf{p'} \in \Lambda^m$ and $\mathbf{q}''=\mathbf{q}-\mathbf{q}' \in \Lambda^n.$ If $\mathbf q''=0,$ from \eqref{equ:ub18} we get that $\|\mathbf p''\| < 1$ for all but finitely many $\mathbf t \in \mathbf T .$ Then we get that $\mathbf p''=0$ (since $\mathbf{p''} \in \Lambda^m$), which is a contradiction as $\alpha \neq \alpha'.$ Hence $\mathbf q'' \neq 0$ and so $\alpha'' \in \mathcal A$. Therefore we have $Y \in \Delta_{\mathbf t}(\alpha'', \phi(\mathbf t)) \subset  H_{\mathbf{t}}(\phi(\mathbf t)).$ This completes the verification of the intersection property with $\phi^*=\phi \in \Phi.$

\subsubsection{The contracting property}
Recall that $\mu$ is a strongly contracting measure on $F^{m \times n}.$ Without loss in generality we may assume that the support $S$ of $\mu$ is bounded. Note that $\mu$ is already doubling, finite and non-atomic. Hence to show that $\mu$ is contracting with respect to $(I,\Phi),$ it only remains to verify  \ref{equ:cont_prop_i} - \ref{equ:cont_prop_iii}. Let $\phi \in \Phi.$ Then $\phi(\mathbf t)=e^{- \tau \sigma(\mathbf t)}$ for some constant $\tau>0.$ Define $\phi^+  :=\sqrt{\phi} \in \Phi.$ Let $r_0$ be as in the definition of strongly contracting measure. Since $\mathbf T$ satisfies \eqref{equ:ub12} and \eqref{equ:p8i}, we have
\begin{equation}\label{equ:ub19}
\phi^+(\mathbf t) \leq \min \{1,r_0\} \quad \text{and}  \quad \sigma(\mathbf t) \geq 0
\end{equation}
for all but finitely many $\mathbf t \in \mathbf T.$ If $\mathbf t=(t_1,\dots,t_{m+n}) \in \mathbf T$ is as above and $\alpha'=(\mathbf p',\mathbf q') \in \mathcal A,$ the set $I_{\mathbf t}(\alpha', \phi(\mathbf t))$ is essentially the set of all $Y \in F^{m \times n}$  such that
\begin{equation}\label{equ:ub20}
  |Y_j \mathbf q'+p_j'+{\theta}_j| < e^{-t_j} \phi(\mathbf t)  \quad \text{and} \quad |q_i'| < e^{t_{m+i}} \phi(\mathbf t)
\end{equation}
for $1 \leq j \leq m$ and $1 \leq i \leq n.$ In a similar fashion, we see that $I_{\mathbf t}(\alpha',\phi^+(\mathbf t))$ is  the set of all $Y \in F^{m \times n}$  such that
\begin{equation}\label{equ:ub21}
  |Y_j \mathbf q'+p_j'+{\theta}_j| < e^{-t_j} \phi^+(\mathbf t)  \quad \text{and} \quad |q_i'| < e^{t_{m+i}} \phi^+(\mathbf t).
\end{equation}
Now we define 
\begin{equation}\label{equ:ub22}
\varepsilon_j=\varepsilon_{j,\mathbf t} := \frac{e^{-t_j} \phi^+(\mathbf t)}{\|\mathbf q'\|} \quad \text{for} \quad j=1,\dots,m \quad \text{and} \quad \delta=\delta_{\mathbf t} := \phi^+(\mathbf t).
\end{equation}
Equation \eqref{equ:p6i} and $\sigma(\mathbf t) \geq 0$ implies that $\sum_{j=1}^{m} t_j \geq 0.$ Hence there exist some $ l \in \{1,\dots, m\}$ such that $t_l \geq 0.$ Therefore, we have
\begin{equation}\label{equ:ub23}
\min_{1 \leq j \leq m} \varepsilon_{j,\mathbf t} < r_0 \quad \text{and} \quad \delta_{\mathbf t} <1,
\end{equation}
since  $\|\mathbf q'\| \geq 1$ and \eqref{equ:ub19} holds.

\noindent Note that $\delta \varepsilon_j=\frac{e^{-t_j}\phi(\mathbf t)}{\|\mathbf q'\|}$ for all $j=1,\dots,m.$ Let $\mathcal{L}_{\mathbf b,\mathbf c}^{\bm{(\varepsilon)}}$ and $\mathcal{L}_{\mathbf b,\mathbf c}^{( \delta \bm{\varepsilon})}$ are given by \eqref{equ:cont_plane} with $\mathbf c:= X^{-(\max_{1 \leq i \leq n}\deg q_i')} \mathbf q'$ and $\mathbf b:=X^{-(\max_{1 \leq i \leq n}\deg q_i')} (\mathbf q' + \boldsymbol{\theta}),$ where $\mathbf q'=(q_1',\dots,q_n').$ Then in view of \eqref{equ:ub20} and \eqref{equ:ub21}, it follows that 
\begin{equation}\label{equ:ub24}
I_{\mathbf t}(\alpha',\phi(\mathbf t))= \mathcal{L}_{\mathbf b,\mathbf c}^{( \delta \bm{\varepsilon})} \quad \text{and} \quad I_{\mathbf t}(\alpha',\phi^+(\mathbf t))= \mathcal{L}_{\mathbf b,\mathbf c}^{( \bm{\varepsilon})}.
\end{equation}
Since $\mu$ is strongly contracting, $\forall \,\,  Y \in \mathcal{L}_{\mathbf b,\mathbf c}^{( \delta \bm{\varepsilon})} \cap S $ there exists an open ball $B_Y$ centred at $Y$ satisfying \eqref{equ:cm_pi} and \eqref{equ:cm_pii}. Following the notation from Section \ref{section:itp}, define $C_{\mathbf t,\alpha'}$  to be the collection of all such balls. Also, define
\begin{equation}\label{equ:ub25}
k_{\mathbf t}:= C (\phi^+)^{\alpha},
\end{equation}
where $C$ and $\alpha$ are constants as in the definition of strongly contracting measure. By \eqref{equ:ub12}, $\sum_{\mathbf t \in  \mathbf T} k_{\mathbf t} < \infty.$ Observe that \ref{equ:cont_prop_i} follows from the definition of $C_{\mathbf t, \alpha'}.$ Finally \ref{equ:cont_prop_ii} and \ref{equ:cont_prop_iii} follows from \eqref{equ:cm_pi}, \eqref{equ:cm_pii} and \eqref{equ:ub24}. This shows that $\mu$ is contracting with respect to $(I, \Phi).$

\begin{remark}
To prove Part \ref{thm:nonextremal_upper_bound_i} of Theorem \ref{thm:nonextremal_upper_bound}, we need to define $\mathcal{U}_{m,n}^{\boldsymbol{\theta}}(\eta):=\{Y \in F^{m \times n}: \omega(Y,\boldsymbol{\theta})> \eta \}$ and show that $\mu(\mathcal{U}_{m,n}^{\boldsymbol{\theta}}(\eta))=0$ for all $\boldsymbol{\theta} \in F^m.$ The Proposition \ref{prop:main} remains unchanged and the only change needed in the proof of Proposition \ref{prop:main} is the definition of $\mathbf T.$ For $u \in \mathbb Z_+$ and $v \in \mathbb Z_+,$ let $\mathbf u:=(u,\dots,u) \in \mathbb Z_+^m$, $\mathbf v=(v_1,\dots,v_n) \in \mathbb Z_+^n$ and $\xi=\frac{mu-\eta nv}{m+\eta n}.$ Now given $u$ and $v$ as above, define $\mathbf t=(t_1,\dots,t_{m+n})$ by
\begin{equation}\label{equ:p28}
\mathbf t :=(u-[\xi],\dots,u-[\xi],v+[\xi],\dots,v+[\xi]).
\end{equation}
Define
\begin{equation}\label{equ:p29}
\mathbf{T}:=\{\mathbf t \in \mathbb Z^{m+n} \,\, \text{given by} \,\, \eqref{equ:p28}: u,v \in \mathbb Z_+ \,\, \text{with} \,\, mu \geq \eta n v  \}.
\end{equation}
Now, in this case, one can easily prove an analogous version of Proposition \ref{prop:main}  by appropriately modifying the arguments of the proof of Proposition \ref{prop:main}.  Once we have Proposition \ref{prop:main}, we can analogously apply the inhomogeneous transference principle to get our desired result. This also completes the proof of Part \ref{thm:main_A} of  Theorem \ref{thm:main}.
\end{remark}

\end{document}